\documentclass[12pt]{amsart}
\pdfoutput=1
\usepackage{amsmath,amssymb,amscd,graphicx}
\usepackage[colorlinks=true, pdfstartview=FitV, linkcolor=blue, citecolor=blue, urlcolor=blue]{hyperref}
\usepackage{mathpazo}
\theoremstyle{plain}
\newtheorem*{nonumtheorem}{Theorem}
\newtheorem{theorem}{Theorem}
\newtheorem*{lemma}{Lemma}

\newtheorem*{conjecture}{Conjecture}

\newcommand{\tre}{\text{Re}}
\newcommand{\tim}{\text{Im}}

\theoremstyle{definition}
\newtheorem*{remark}{Remark}

\newtheorem*{definition}{Definition}

\begin{document}
\title{Lehmer Pairs Revisited}
\author{Jeffrey Stopple}
\begin{abstract}
We seek to understand how the technical definition  of a Lehmer pair can be related to more analytic properties of the Riemann zeta function, particularly the location of the zeros of $\zeta^\prime(s)$.  Because we are interested in the connection \cite{CSV} between Lehmer pairs and the de Bruijn-Newman constant $\Lambda$, we assume the Riemann Hypothesis throughout.  We define strong Lehmer pairs via an inequality on the derivative  of the pre-Schwarzian of Riemann's function $\Xi(t)$, evaluated at consecutive zeros:
\[
-\Delta^2\left(P\Xi^\prime(\gamma_+)+P\Xi^\prime(\gamma_-)\right)<42/5.
\]
Theorem \ref{Th:1} shows that strong Lehmer pairs are Lehmer pairs.  Theorem \ref{Th:2} describes  $P\Xi^\prime(\gamma)$ in terms of  $\zeta^\prime(\rho)$  where $\rho=1/2+i\gamma$.  Theorem \ref{Th:3} expresses $P\Xi^\prime(\gamma_+)+P\Xi^\prime(\gamma_-)$ in terms of nearby zeros $\rho^\prime$ of $\zeta^\prime(s)$.  We examine $114\,661$ pairs of zeros of $\zeta(s)$ around height $t=10^6$, finding $855$ strong Lehmer pairs.  These are compared to the corresponding zeros of $\zeta^\prime(s)$ in the same range.
\end{abstract}
\email{stopple@math.ucsb.edu}

\maketitle

\section{Introduction}

In 1927 Polya introduced a deformation parameter in Riemann's function $\Xi(t)$, so the zeros flow according to the (backward) heat equation.  The de Bruijn-Newman constant $\Lambda$ measures the tendency of any deformation backward in time to result in non-real zeros.  Newman conjectured that $\Lambda\ge0$, which says in crude terms that  the Riemann Hypothesis (if true) is as close to failing as it possibly could be.  Sufficiently close pairs of zeros of the Riemann zeta function, so-called Lehmer pairs,  may be used to give lower bounds on $\Lambda$, and in fact the existence of infinitely many Lehmer pairs would imply that $\Lambda\ge 0$.  This is explained in more detail in \S \ref{S:2} and \S \ref{S:3}, which are expository and may be skipped by experts.

We seek to understand how the technical definition \cite{CSV} of a Lehmer pair can be related to more analytic properties of the Riemann zeta function, particularly the location of the zeros of $\zeta^\prime(s)$.  

In \S \ref{S:4}, we define strong Lehmer pair via an inequality on the derivative  of the pre-Schwarzian of  $\Xi(t)$, evaluated at consecutive zeros:
\[
-\Delta^2\left(P\Xi^\prime(\gamma_+)+P\Xi^\prime(\gamma_-)\right)<42/5.
\]
Theorem \ref{Th:1} in \S \ref{S:4}  shows that strong Lehmer pairs are Lehmer pairs.  In \S \ref{S:5}, Theorem \ref{Th:2} describes  $P\Xi^\prime(\gamma)$ in terms of  $\zeta^\prime(\rho)$  where $\rho=1/2+i\gamma$ is a zero of $\zeta(s)$.  Theorem \ref{Th:3} in \S \ref{S:5} expresses $P\Xi^\prime(\gamma_+)+P\Xi^\prime(\gamma_-)$ in terms of nearby zeros $\rho^\prime$ of $\zeta^\prime(s)$.  In \S \ref{S:6} is a discussion of the contributions of the various terms, in particular, the contribution of the (typically) unique zero $\rho_0^\prime$ of $\zeta^\prime(s)$ whose imaginary part lies between those of consecutive zeros $\rho_-$ and $\rho_+$ of $\zeta(s)$.
In \S \ref{S:7} we examine $114\,661$ pairs of zeros of $\zeta(s)$ around height $t=10^6$, finding $855$ strong Lehmer pairs.  These are compared to the corresponding zeros of $\zeta^\prime(s)$ in the same range.  Finally, we conjecture that Soundararajan's Conjecture B in \cite{Sound} implies that $\Lambda=0$.

Because we are interested in the connection between Lehmer pairs and the de Bruijn-Newman constant $\Lambda$, the Riemann Hypothesis $\Lambda\le 0$ is assumed  throughout this paper.

\section{The de Bruijn-Newman constant}\label{S:2}

The Riemann xi functions $\xi(s)$ and $\Xi(t)$ are defined in Titchmarsh  \cite[p.\ 16]{Tit}, with $s=1/2+it$, as
\begin{align*}
\xi(s)=&\frac12s(s-1)\pi^{-s/2}\Gamma(s/2)\zeta(s)\overset{\text{def.}}=h(s)\zeta(s),\\
\Xi(t)=&\xi(1/2+it)
=2\int_0^\infty \Phi(u)\cos(ut) \, du,
\end{align*}
where $\Phi(u)$ is defined as \cite[\S 10.1]{Tit}
\[
2\sum_{n=1}^\infty\left(2\pi^2n^4\exp(9u/2)-3\pi n^2\exp(5u/2)\right)\exp(-n^2\pi\exp(2u)).
\]
Since we are going to introduce the heat equation we have a clash of notations: $t=\tim(s)$ for the Riemann zeta function v. $t$ representing time in the heat equation.  So in this section only we will write $\Xi(x)$ instead of $\Xi(t)$.  In 1927 Polya \cite{Polya} introduced a deformation parameter $t$:
\[
\Xi_t(x)=\int_0^\infty \exp(t u^2)\Phi(u)\cos(ux)\, du,
\]
so that for $t=0$, $\Xi_0(x)$ is just $\Xi(x)/2$.  It is easy to see that this function satisfies the \textsc{backward heat equation}
\begin{equation}\label{Eq:bheat}
\frac{\partial \Xi}{\partial t}+\frac{\partial^2 \Xi}{\partial x^2}=0.
\end{equation}
(Neither Polya, nor de Bruijn, nor Newman call attention to this.)
\begin{nonumtheorem}[de Bruijn, \cite{dB}]  The following are true about the zeros in $x$:
\begin{enumerate}
\item For $t\ge 1/2$, $\Xi_t(x)$ has only real zeros.
\item If for some real $t$, $\Xi_t(x)$ has only real zeros, then $\Xi_{t^\prime}(x)$ also has only real zeros for any $t^\prime>t.$
\end{enumerate}
\end{nonumtheorem}
In 1976 Newman \cite{Newman} showed that
\begin{nonumtheorem}[Newman]
There exists a real constant $\Lambda$, $-\infty<\Lambda\le 1/2$, such that
\begin{enumerate}
\item $\Xi_t(x)$ has only real zeros if and only if $t\ge\Lambda$, and
\item $\Xi_t(x)$ has some complex zeros if $t<\Lambda$.
\end{enumerate}
\end{nonumtheorem}
The constant $\Lambda$ is known as the de Bruijn-Newman constant.  The Riemann hypothesis is the conjecture that $\Lambda\le 0$.  Newman made the complementary conjecture that $\Lambda\ge 0$, with the often quoted remark
\begin{quotation}
\lq\lq\emph{This new conjecture is a quantitative version of the dictum that the Riemann hypothesis, if true, is only barely so.}\rq\rq
\end{quotation}

Since we are interested here in the Newman conjecture, we can assume the Riemann Hypothesis $\Lambda\le 0$, since its negation is $\Lambda>0$ which implies the Newman conjecture.

Given the significance of the de Bruijn-Newman constant, much work has gone into estimating lower bounds, and a decade saw twelve orders of magnitude improvement, from the first bound \cite{CVN2}
\[
-50<\Lambda
\]
to the current record \cite{Saouter}
\[
-1.14\times 10^{-11}<\Lambda.
\]
A breakthrough occurred in the work of Csordas, Smith and Varga, \cite{CSV}, in which it was realized that unusually close pairs of zeros of the Riemann zeta function, the so-called Lehmer pairs, could be used to give lower bounds on $\Lambda$.  

\section{Lehmer pairs}\label{S:3}
Let as usual $Z(t)$ denote the Hardy function. 
Assuming the Riemann Hypothesis, it is well known that for $t$ sufficiently large, $Z^\prime/Z(t)$ is monotonically decreasing between consecutive zeros of $Z(t)$.  Consequently, for $t$ sufficiently large, $Z(t)$ has no negative local maximum or positive local minimum.

\begin{figure}
\begin{center}
\includegraphics[scale=.8, viewport=0 0 400 230,clip]{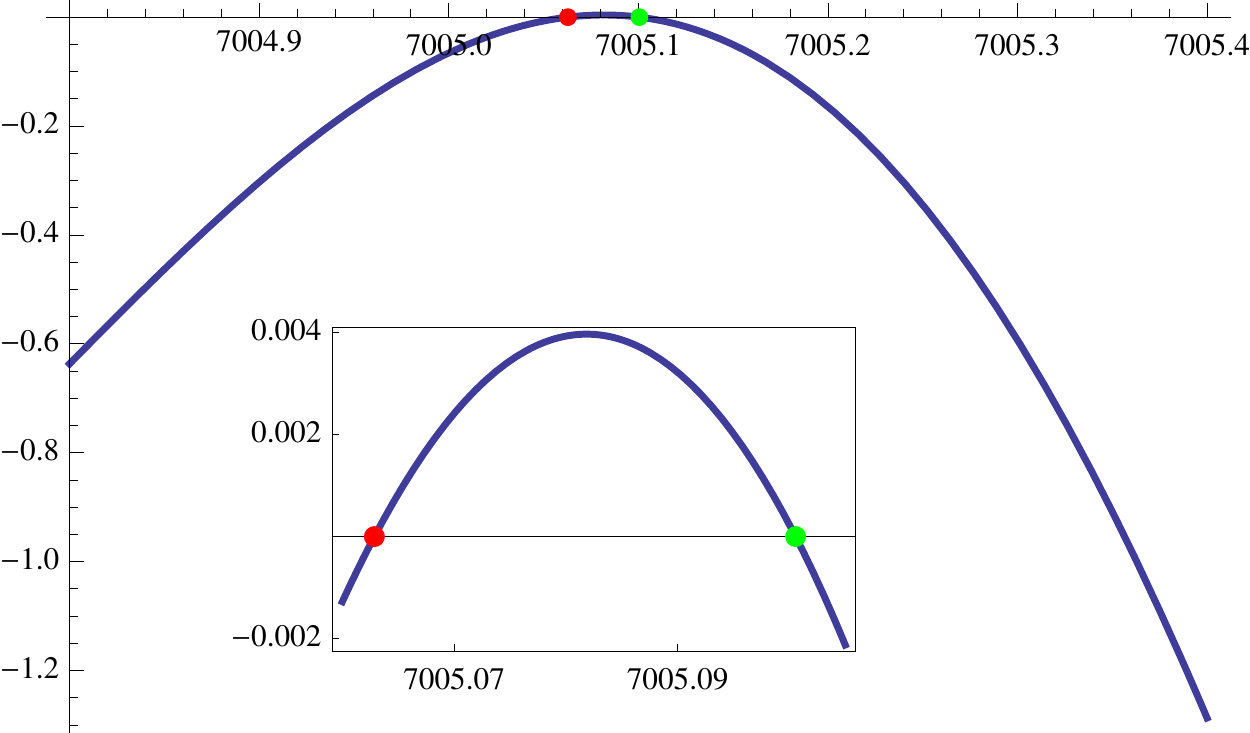}
\caption{Lehmer's phenomenon for $Z(t)$, showing the Lehmer pair $\{\gamma_{6709},\gamma_{6710}\}$ in red and green.}\label{F:1}
\end{center}
\end{figure}

Lehmer \cite{Lehmer} discovered a \lq near-counterexample\rq\  in the form of a very small positive local maximum of $Z(t)$, between two unusually close zeros, near $t=7005.$, see Figure \ref{F:1}.  This and other such examples are known informally as \lq\lq Lehmer's phenomenon,\rq\rq\ with the corresponding close pair of zeros called a \lq\lq Lehmer pair.\rq\rq\ \ 

In \cite{CSV} Csordas, Smith, and Varga give a precise, though somewhat technical, definition of Lehmer pair:
\begin{definition}  Let $0<\gamma_-<\gamma_+$ be two consecutive simple positive zeros of $\Xi(t)$ (so also $Z(\gamma_-)=0=Z(\gamma_+)$).   Let $\Delta=\gamma_+-\gamma_-$, and let 
\begin{equation}\label{Eq:gkdef}
g=\sum_{\gamma\ne\gamma_-,\gamma_+}\frac{1}{(\gamma-\gamma_-)^2}+\frac{1}{(\gamma-\gamma_+)^2}.
\end{equation}
Then $\{\gamma_-,\gamma_+\}$ is a \textsc{Lehmer pair} if
\begin{equation}\label{Def:Lehmer}
\Delta^2\, g<\frac{4}{5}.
\end{equation}
\end{definition}
They prove the following
\begin{nonumtheorem}[Csordas \emph{et al}]  Let $\{\gamma_-,\gamma_+\}$ be a Lehmer pair, and define
\begin{equation}\label{Def:lambdak}
\lambda=\frac{(1-5\Delta^2 g/4)^{4/5}-1}{8g}
\end{equation}
so that $-1/(8g)<\lambda<0$.  Then the de Bruijn-Newman constant $\Lambda$ satisfies
\begin{equation}\label{Eq:Lambdabound}
\lambda\le \Lambda.
\end{equation}
Furthermore, the existence of infinitely many Lehmer pairs implies that the de Bruijn-Newman constant $\Lambda$ is equal $0$.
\end{nonumtheorem}

Odlyzko, in \cite{Odlyzko}, gives a heuristic argument that the GUE distribution of the Riemann zeros should imply the existence of infinitely many Lehmer pairs, and thus $\Lambda=0$.  With the eventual goal of making this heuristic rigorous, we seek to understand how the technical definition of Lehmer pair can be related to more analytic properties of $\zeta(s)$, in particular the location of the zeros of $\zeta^\prime(s)$.  More precisely, we note that the definition (\ref{Def:Lehmer}) is the least restrictive possible that allows the proof of (\ref{Eq:Lambdabound}) to go through.  As a result, one finds that Lehmer pairs are not particularly rare.  A computation shows there are 29 examples amongst the first 649 zeros ($t<1000$). But from the point of view of proving there are infinitely many Lehmer pairs, a more restrictive definition may be more useful.

\section{The pre-Schwarzian}\label{S:4}

The logarithmic derivative of the derivative of an analytic function $f(z)$ is sometimes called the pre-Schwarzian of $f(z)$:
\[
Pf(z)\overset{\text{def.}}=\frac{f^{\prime\prime}(z)}{f^\prime(z)}.
\]
We make the notational convention that $P$ has priority over derivative:
\[
Pf^\prime(z) \quad\text{means}\quad \left(\frac{f^{\prime\prime}(z)}{f^\prime(z)}\right)^\prime,\quad\text{not}\quad \frac{f^{\prime\prime\prime}(z)}{f^{\prime\prime}(z)}.
\]
\begin{lemma}
Suppose $q(z)$ is any analytic function in a domain $D$, and $q(w)\ne0$.  Then for $f(z)=(z-w)q(z)$ we have that, when evaluated at $w$, the pre-Schwarzian satisfies
\begin{equation}\label{Eq:pre}
Pf(w)=2\frac{q^\prime(w)}{q(w)}.
\end{equation}
(This is Lemma 2.3 in \cite{CSV}.)\ \ 
Furthermore, one calculates that
\begin{equation}\label{Eq:prepre}
Pf^\prime(w)=\frac{3q\cdot q^{\prime\prime}-4(q^\prime)^2}{q^2}(w)=3\left(\frac{q^\prime}{q}\right)^\prime(w)-\left(\frac{q^\prime}{q}\right)^2(w),
\end{equation}
and so
\begin{equation}\label{Eq:prepre1}
(Pf^\prime+\frac{1}{4}(Pf)^2)(w)=3\left(\frac{q^\prime}{q}\right)^\prime(w).
\end{equation}
\end{lemma}

\begin{definition}  We  say that $\{\gamma_-,\gamma_+\}$ is a \textsc{strong Lehmer pair} if 
\begin{equation}\label{Eq:SLP}
\Delta^2\left(-P\Xi^\prime(\gamma_+)-P\Xi^\prime(\gamma_-)\right)<\frac{42}{5}.
\end{equation}
\end{definition}

\begin{theorem}\label{Th:1}  The above inequality \eqref{Eq:SLP} is a sufficient condition for $\{\gamma_-,\gamma_+\}$ to be a Lehmer pair.
\end{theorem}
\begin{proof}
We have
\[
P\Xi(t)=iP\xi(1/2+it),\quad\text{and}\quad -P\Xi^\prime(t)=P\xi^\prime(1/2+it).
\]
From the Hadamard factorization theorem we have that
\begin{equation}\label{Eq:Hadamard}
\Xi(t)=\Xi(0)\prod_{j=1}^\infty\left(1-\frac{t^2}{\gamma_j^2}\right).
\end{equation}
Fixing $\gamma_+$ we have that
\[
\Xi(t)=(t-\gamma_+)q(t),
\]
where
\[
q(t)=-\frac{\Xi(0)}{\gamma_+}\cdot\left(1+\frac{t}{\gamma_+}\right)\prod_{\substack{j=1\\\gamma_j\ne \gamma_+}}^\infty\left(1-\frac{t^2}{\gamma_j^2}\right).
\]
We  write
\begin{align*}
G(t)\overset{\text{def.}}=&
\frac{\left(3\Xi^{\prime\prime}(t)\right)^2-4\Xi^{\prime\prime\prime}(t)\cdot \Xi^{\prime}(t)}{4\Xi^{\prime}(t)^2}\\
=&-\left(P\Xi^\prime+\frac{1}{4}(P\Xi)^2\right)(t).
\end{align*}
From \eqref{Eq:prepre1} with $w=\gamma_+$ , and analogously with $w=\gamma_-$ we deduce that
\[
G(\gamma_+)+G(\gamma_-)=
\sum_{\gamma_j\ne \gamma_+}\frac{3}{(\gamma_j-\gamma_+)^2}+
\sum_{\gamma_j\ne \gamma_-}\frac{3}{(\gamma_j-\gamma_-)^2}=\frac{6}{\Delta^2}+3g.
\]
Upon rearranging we have that
\[
3\Delta^2g+6=\Delta^2\left(G(\gamma_+)+G(\gamma_-)\right),
\]
and the condition to have a Lehmer pair becomes
\[
\Delta^2\left(G(\gamma_+)+G(\gamma_-)\right)<\frac{42}{5}.
\]
Since $P\Xi(t)$ is real valued for real $t$, $(P\Xi(t))^2$ is positive, and so
\[
G(t)<-P\Xi^\prime(t).
\]
\end{proof}

\section{The zeros of $\zeta^\prime(s)$}\label{S:5}
Figure \ref{F:2} shows the argument of $\zeta^\prime(s)/\zeta(s)$, interpreted as a color\footnote{red - yellow - green - blue, as one travels around the origin counterclockwise.}, in a region which includes Lehmer's example.  The Riemann zeros $1/2+i\gamma_{6709}$ and $1/2+i\gamma_{6710}$ are now poles, while in between we see a zero of $\zeta^\prime(s)$ at 
$0.50062354 + 7005.08185555i$,
very close to the critical line, even on the scale of this close pair of Riemann zeros.  (To see the distinction between poles and zeros, observe the orientation of the colors.)\ \   Another Lehmer pair in the literature, $\{\gamma_{1048449114},\gamma_{1048449115}\}$ was discovered by van de Lune, te Riele, and Winter (see \cite{COSV}).  The corresponding zero of $\zeta^\prime(s)$ lies at 
\[
0.500000013216794+ 
3.888588860023394\cdot10^8 i.\label{vdL}
\]

\begin{figure}
\begin{center}
\includegraphics[scale=1.5, viewport=0 0 150 150,clip]{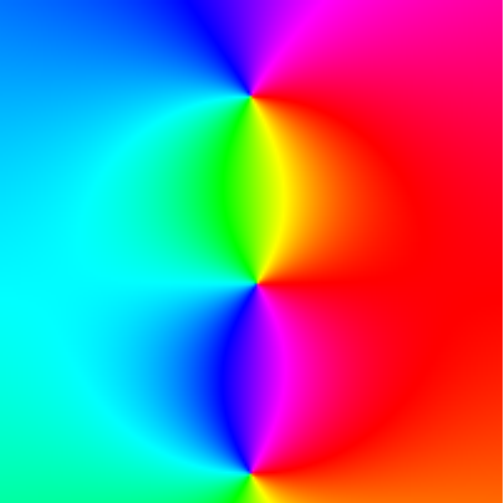}
\caption{Lehmer's example: $\arg(\zeta^\prime(s)/\zeta(s))$ for $0.475\le\sigma\le0.525$, and $7005.6\le t\le7005.11$.}\label{F:2}
\end{center}
\end{figure}

We would like to relate the zeros of $\Xi^\prime$ to those of $\zeta^\prime$.  
With $h(s)=\pi^{s/2}\Gamma(s/2)$, let $\eta(s)=h(s)\zeta^\prime(s)$; the zeros of $\zeta^\prime$ are the zeros of $\eta$.  
\begin{theorem}\label{Th:2}  For a zero $\rho=1/2+i\gamma$ of $\zeta(s)$,
\begin{multline}\label{Eq:PXi}
-P\Xi^\prime(\gamma)=
\tre\left(\frac{\eta^\prime}{\eta}\right)^\prime(\rho)
+\tre^2\left(\frac{h^\prime}{h}(\rho)\right)\\
+\tim\left(\frac{\eta^\prime}{\eta}(\rho)\right)\tim\left(\frac{h^\prime}{h}(\rho)\right)
+2\tre\left(\left(\frac{h^\prime}{h}\right)^\prime(\rho)\right).
\end{multline}
\end{theorem}
\begin{proof}
From
\[
\xi^\prime(s)=\eta(s)\left(\frac{h^\prime(s)\zeta(s)}{h(s)\zeta^\prime(s)}+1\right),
\]
we have
\begin{multline*}
P\xi(s)=\frac{\eta^\prime(s)}{\eta(s)}+\frac{h'(s) \zeta '(s)}{\zeta (s) h'(s)+h(s) \zeta '(s)}\\
-\zeta (s)\cdot \frac{h(s) h'(s) \zeta ''(s)+\zeta '(s) \left(h'(s)^2-h(s) h''(s)\right)}{h(s) \zeta '(s) \left(\zeta (s)
   h'(s)+h(s) \zeta '(s)\right)}.
\end{multline*}
When evaluated at $\rho=1/2+i\gamma$ this simplifies to
\[
P\xi(\rho)=\frac{\eta^\prime}{\eta}(\rho)+\frac{h^\prime}{h}(\rho).
\]
Similarly, (and with the help of \emph{Mathematica}),
\begin{equation}\label{Eq:mess}
P\xi^\prime(\rho)=\left(\frac{\eta^\prime}{\eta}\right)^\prime(\rho)-\frac{\eta^\prime}{\eta}(\rho)\frac{h^\prime}{h}(\rho)
+2\left(\frac{h^\prime}{h}\right)^\prime(\rho).
\end{equation}

Since $P\xi$ is purely imaginary on the critical line,
\[
\tre\left(\frac{\eta^\prime}{\eta}(\rho)\right)=-\tre\left(\frac{h^\prime}{h}(\rho)\right).
\]
Meanwhile $-P\Xi^\prime(\gamma)=P\xi^\prime(\rho)$ is real.  Taking real parts in \eqref{Eq:mess} gives the result.
\end{proof}

The following lemma, approximating the logarithmic derivative of a function $f$ by a sum over its zeros $\rho^\prime$,  is an adaptation of  \cite[Lemma ($\alpha$)]{Tit} for our purposes, with the constants made explicit.  
\begin{lemma}[$\alpha$]
Let $s_0$ in $\mathbb C$, and $\mathcal D$ denote the disk of radius $r$ centered at $s_0$.  Let $\mathcal R$ denote a rectangle in $\mathbb C$ with the property that all of the zeros $\rho^\prime$ in $\mathcal D$ of $f$ belong to $\mathcal R$:
\[
f(\rho^\prime)=0\text{ and }\rho^\prime\in \mathcal D\Rightarrow \rho^\prime\in \mathcal R.
\]
Fix
\[
B\ge\max_{\rho^\prime \in\mathcal R}|s_0-\rho^\prime|,
\]
and let
\[
\Gamma=\partial\left\{s\in \mathbb C\,|\, \text{dist}(s,\mathcal R\cup\mathcal D)\le B\right\}.
\]
If $f(s)$ is regular and
\[
\left|\frac{f(s)}{f(s_0)}\right| < \exp(M)
\]
for $s$ on $\Gamma$, then for $|s-s_0|\le 2r/3$
\begin{gather*}
\left|\frac{f^\prime(s)}{f(s)}-\sum_{\rho^\prime\in\mathcal R} \frac{1}{s-\rho^\prime}\right| <\frac{30M}{r},\\
\left|\left(\frac{f^\prime(s)}{f(s)}\right)^\prime+\sum_{\rho^\prime\in\mathcal R} \frac{1}{(s-\rho^\prime)^2}\right| <\frac{90M}{r^2}.
\end{gather*}
\end{lemma}
\begin{proof} The function
\[
g(s)=f(s)\prod_{\rho^\prime\in\mathcal R}(s-\rho^\prime)^{-1}
\]
is regular inside $\Gamma$ and not zero inside $\mathcal D$.  On  $\Gamma$, $|s-\rho^\prime|\ge B\ge |s_0-\rho^\prime|$, so
\[
\left|\frac{g(s)}{g(s_0)}\right|=\left|\frac{f(s)}{f(s_0)}\prod\left(\frac{s_0-\rho^\prime}{s-\rho^\prime}\right)\right|\le \left|\frac{f(s)}{f(s_0)}\right|<\exp(M).
\]
This inequality therefore holds inside $\Gamma$ also.  Hence the function
\[
h(s)=\log\left(\frac{g(s)}{g(s_0)}\right),
\]
where the logarithm is zero at $s=s_0$, is regular for $s\in\mathcal D$, and
\[
h(s_0)=0,\qquad\tre(h(s))<M.
\]
By the Borel-Carath\'eodory Theorem, for $|s-s_0|\le 5r/6$
\[
|h(s)| \le \frac{2\cdot 5r/6}{r-5r/6}M = 10M,
\]
and so, for $|s-s_0|\le 2r/3$,
\begin{gather*}
|h^\prime(s)|=\left|\frac{1}{2\pi i}\int_{|z-s|=r/3}\frac{h(z)}{(z-s)^2}\, dz\right|\le \frac{30M}{r},\\
|h^{\prime\prime}(s)|=\left|\frac{2}{2\pi i}\int_{|z-s|=r/3}\frac{h(z)}{(z-s)^3}\, dz\right|\le \frac{90M}{r^2}.
\end{gather*}
\end{proof}

The following is a slight generalization of \cite[(14.15.2)]{Tit}.
\begin{lemma}  Let  $s=1/2+i t$ with $|t-t_0|\le1/\log(t_0)$.
Then
\begin{gather}
\frac{\eta^\prime(s)}{\eta(s)}=\sum_{|\gamma^\prime-t_0|\le 1/\log\log(t_0)}\frac{1}{s-\rho^\prime}+O(\log t_0),\label{Eq:impart}\\
\left(\frac{\eta^\prime(s)}{\eta(s)}\right)^\prime=-\sum_{|\gamma^\prime-t_0|\le 1/\log\log t_0}\frac{1}{(s-\rho^\prime)^2}+O(\log t_0\log\log t_0).\label{Eq:realpart}
\end{gather}
\end{lemma}
\begin{proof}
In Lemma $(\alpha)$ we take $f(s)=\eta(s)$ and
\[
s_0=1/2-1/(\sqrt{3}\log\log t_0)+i t_0,\quad r=2/(\sqrt3\log\log t_0).
\]
We let $\mathcal R$ be the rectangle
\[
\left\{s\in\mathbb C\,|\, 1/2\le \sigma\le 1/2+1/\log\log(t_0), |t-t_0|\le 1/\log\log(t_0)  \right\},
\]
which contains all the zeros of $\zeta^\prime(s)$ inside $\mathcal D$, and we can take $B=2/\log\log(t_0)$.  Via the symmetry of $\mathcal R\cup\mathcal D$ around the horizontal line $t=\gamma_0^\prime$, we see that $\Gamma$ is contained in the vertical strip 
\[
\{1/2-3/(2\log\log(t_0))\le\sigma\le 1/2+3/\log\log(t_0)\}.
\]
In this strip we have
\[
\left|\frac{\Gamma(s/2)}{\Gamma(s_0/2)}\right|\ll t^{\tre(s)-\tre(s_0)}\ll \exp(4\log(t_0)/\log\log(t_0))
\]
by Stirling's formula.  The proof of \cite[Theorem 14.5]{Tit} gives, for some $A>0$,
\[
|\zeta(s)|\ll \exp(A\log(t_0)/\log\log(t_0)),
\]
 and by the Cauchy Integral Formula for $\zeta^\prime$ (taking a circle of radius $1/\log\log(t_0)$ around $s$), gives the same for $|\zeta^\prime(s)|$ inside the curve $\Gamma$.

From the functional equation we have
\begin{equation}\label{Eq:zetapFE}
\zeta^\prime(s_0)=\chi(s_0)\zeta(1-s_0)\left(\frac{\chi^\prime(s_0)}{\chi(s_0)}-\frac{\zeta^\prime(1-s_0)}{\zeta(1-s_0)}\right).
\end{equation}
From Stirling's formula we have
\begin{gather*}
|\chi(s_0)|\sim\left(t_0/2\pi\right)^{1/\sqrt{3}\log\log t_0}\sim\exp(\log(t_0)/(\sqrt{3}\log\log(t_0)),\\
\left|\frac{\chi^\prime(s_0)}{\chi(s_0)}\right|\sim\log(t_0/2\pi),
\end{gather*}
while \cite[Theorem 14.14 B]{Tit} gives for some $A>0$
\[
|\zeta(1-s_0)|\ge\exp(-A\log(t_0)/\log\log(t_0)).
\]
Meanwhile, \cite[(14.14.5)]{Tit} and the Cauchy Integral formula for derivatives (on a circle of radius $1/4\log\log(t_0)$) gives
\[
\left|\frac{\zeta^\prime(1-s_0)}{\zeta(1-s_0)}\right|\ll \log(t_0)^{1/\log\log(t_0)}.
\]
Thus from \eqref{Eq:zetapFE} we deduce that for some $A>0$,
\[
|\zeta^\prime(s_0)|\ge\exp(-A\log(t_0)/\log\log(t_0)),
\]
and we may chose $M=A\log (t_0)/\log\log(t_0)$ in Lemma ($\alpha$).
\end{proof}

Meanwhile, Stirling's formula gives that
\begin{equation}\label{Eq:Stirling1}
\frac{h^\prime}{h}(\rho)=\frac{1}{2}\log\left(\frac{\gamma}{2\pi}\right)+\frac{\pi}{4}i+O\left(\frac{1}{\gamma}\right),\qquad\left(\frac{h^\prime}{h}\right)^\prime(\rho)\ll\frac{1}{\gamma}.
\end{equation}

Let  $\rho_0^\prime=\beta_0^\prime+i\gamma_0^\prime$ a fixed zero of $\zeta^\prime(s)$ as above with the added assumption that
\[
\beta_0^\prime<1/2+1/\log(\gamma_0^\prime).
\]
Let $\gamma_-$ and $\gamma_+$ be the imaginary parts of successive zeros $\rho_-$ and $\rho_+$ of $\zeta(s)$ such that
\[
\gamma_-<\gamma_0^\prime<\gamma_+.
\]
By \cite[Proposition 1.6]{Sound}, $\rho_0^\prime$ is the unique zero of $\zeta^\prime(s)$ in the region
\[
1/2<\sigma<1/2+1/\log(\gamma_0^\prime),\qquad \gamma_-<t<\gamma_+.
\]
Plugging \eqref{Eq:impart}, \eqref{Eq:realpart}, and \eqref{Eq:Stirling1} into \eqref{Eq:PXi} we obtain

\begin{theorem}\label{Th:3}  We have
\begin{multline}\label{Eq:newPXi}
-P\Xi^\prime(\gamma_+)=
-\tre\sum_{|\gamma^\prime-\gamma_0^\prime|\le 1/\log\log(\gamma_0^\prime)}\frac{1}{(\rho_+-\rho^\prime)^2}+O\left(\log(\gamma_0^\prime)\right)\\
+\left(\tim\sum_{|\gamma^\prime-\gamma_0^\prime|\le 1/\log\log(\gamma_0^\prime)}\frac{1}{\rho_+-\rho^\prime}+O\left(\log(\gamma_0^\prime)\log\log(\gamma_0^\prime)\right)\right)\left(\frac{\pi}{4}+O\left(\frac{1}{\gamma_0^\prime}\right)\right)\\
+\left(\frac{1}{2}\log\left(\frac{\gamma_0^\prime}{2\pi}\right)+O\left(\frac{1}{\gamma_0^\prime}\right)\right)^2,
\end{multline}
and analogously for $-P\Xi^\prime(\gamma_-)$.
\end{theorem}
\begin{remark}
When \eqref{Eq:newPXi} is multiplied by $\Delta^2$, the error terms are $o(1)$.
\end{remark}
\section{Discussion}\label{S:6}
\subsection*{Series expansions}
To understand the contributions of the various terms in \eqref{Eq:newPXi}, we need to think about the location of the zero $\rho_0^\prime$ of $\zeta^\prime(s)$ relative to $\rho_+$ and $\rho_-$ (notation as above.).  There is a lot in the literature about the relation between zeros of $\zeta(s)$ and the real parts of zeros of $\zeta^\prime(s)$, but not so much about the imaginary parts.

Following the methods of \cite{Duenez}, we will develop an argument to support our expectation that when 
$(\gamma_+-\gamma_-)\log(\gamma_0^\prime)/(2\pi)$,
the normalized gap,  is small (notation as above), then 
\[
\gamma_0^\prime\approx (\gamma_++\gamma_-)/2.
\]
We define $t_0$, $X$, $Y$, and $\lambda$ via\footnote{in \cite[\S 6]{Duenez}, $Y$ is \emph{assumed} to be $0$.}
\begin{gather}\label{Eq:Dueneznotation}
\rho_{\pm}=1/2+i(t_0\pm \Delta/2),\qquad
\rho_0^\prime=1/2+X+i(t_0+Y),\\
\lambda\overset{\text{def.}}=\log(t_0/2\pi),
\end{gather}
so $t_0=(\gamma_++\gamma_-)/2$, $\beta_0^\prime=1/2+X$, $\gamma_0^\prime=t_0+Y$.
We have
\begin{multline}
\frac{\zeta^\prime}{\zeta}(\rho_0^\prime)=0=-\lambda/2+\frac{1}{\rho_0^\prime-\rho_+}+\frac{1}{\rho_0^\prime-\rho_-}\\
+\sum_{\rho\ne\rho_\pm}\left(\frac{1}{\rho_0^\prime-\rho}-\frac{1}{\rho}\right)
+\text{const.}+O(1/t_0).\label{Eq:zetapzeta}
\end{multline}
(In fact the imaginary part of the constant is $-\pi/4$, a fact we will use below.)\ \ 
We  rescale with 
\[
x=X\lambda,\qquad y=Y\lambda,\qquad\delta=\Delta\lambda/2\pi.
\]
The contribution in \eqref{Eq:zetapzeta} of the zeros $\rho_\pm$ to the real part is
\[
\frac{2 x \lambda \left(\pi ^2 \delta ^2+x^2+y^2\right)}{\left(\pi ^2 \delta ^2-2 \pi  \delta  y+x^2+y^2\right) \left(\pi ^2 \delta ^2+2 \pi 
   \delta  y+x^2+y^2\right)},
\]
 to the imaginary part is
\[
\frac{-2 y \lambda \left(-\pi ^2 \delta ^2+x^2+y^2\right)}{\left(\pi ^2 \delta ^2-2 \pi  \delta  y+x^2+y^2\right) \left(\pi ^2 \delta ^2+2
   \pi  \delta  y+x^2+y^2\right)},
\]
and each term in the sum over $\rho$ contributes
\[
\frac{x\lambda}{\left(y+(t_0-\gamma)\lambda\right)^2+x^2}+\frac{1/2}{1/4+\gamma^2}
\]
to the real part and
\[
\frac{\lambda ((\gamma -t_0) \lambda-y)}{\left(y+(t_0-\gamma)\lambda\right)^2+x^2}+\frac{\gamma}{1/4+\gamma^2}
\]
to the imaginary part.
We multiply \eqref{Eq:zetapzeta} by
\[
\left(\pi ^2 \delta ^2-2 \pi  \delta  y+x^2+y^2\right) \left(\pi ^2 \delta ^2+2
   \pi  \delta  y+x^2+y^2\right)/\lambda,
\]
and expand $x$ and $y$ as functions of $\delta$:
\begin{gather*}
x(\delta)=\frac{x^{\prime\prime}(0)}{2}\delta^2+O(\delta^4)\\
y(\delta)=\frac{y^{\prime\prime}(0)}{2}\delta^2+O(\delta^4).
\end{gather*}
The real parts of the contribution of the zeros $\rho_\pm$, as well as the terms which originated from the $-\lambda/2$ and from the constant, is now
\[
\left(\frac{\pi^4\log(\pi)}{2\lambda}-\frac{\pi^4}{2}+\pi^2x^{\prime\prime}(0)\right)\delta^4+O(\delta^6),
\]
while the remaining zeros contribute only $O(\delta^6)$ to (the right side of) \eqref{Eq:zetapzeta}.  From (the left side of) \eqref{Eq:zetapzeta} we deduce that every coefficient of $\delta$ must be $0$, in particular,
\begin{gather*}
\frac{\pi^4\log(\pi)}{2\lambda}-\frac{\pi^4}{2}+\pi^2x^{\prime\prime}(0)=0\qquad\text{or}\\
x(\delta)=\frac{\pi^2}{4}\left(1-\frac{\log(\pi)}{\lambda}\right)\delta^2+O(\delta^4),
\end{gather*}
where the coefficient of $\delta^2$ is $\sim\pi^2/4$ just as in \cite[(6.15)]{Duenez}.  

For the example of van de Lune \emph{et.\ al.}\ on page \pageref{vdL}, we compute
\[
x=2.371205\cdot 10^{-7},\quad \frac{\pi^2}{4}\left(1-\frac{\log(\pi)}{\lambda}\right)\delta^2=2.219909\cdot 10^{-7}.
\]

In contrast, the imaginary parts of the contribution of the zeros $\rho_\pm$ and the term which originated from the constant is
\[
\left(-\frac{\pi^5}{4\lambda}+\pi^2 y^{\prime\prime}(0)\right)\delta^4+O(\delta^6),
\]
while the remaining zeros contribute
\[
-\frac{\pi^4}{\lambda}\sum_{\rho\ne\rho_\pm}\frac{1}{t_0-\gamma}\,\,\delta^4+O(\delta^6).
\]
Again we determine
\begin{gather*}
\pi^2 y^{\prime\prime}(0)-\frac{\pi^5}{4\lambda}-\frac{\pi^4}{\lambda}\sum_{\rho\ne\rho_\pm}\frac{1}{t_0-\gamma}=0\qquad\text{or}\\
y(\delta)=\frac{\pi^2}{2\lambda}\left(\frac{\pi}{4}+\sum_{\rho\ne\rho_\pm}\frac{1}{t_0-\gamma}\right)\,\, \delta^2+O(\delta^4).
\end{gather*}
Here the $\delta^2$ coefficient, although not identically $0$, will only be significant when a failure of cancellation in the sum over $\rho\ne\rho_\pm$ is able to dominate the $\lambda^{-1}\sim\log(t_0)^{-1}$ decay.  For the example of van de Lune \emph{et.\ al.}\  we have
$y=7.217276\cdot 10^{-9}$.
\subsection*{Contributions of the terms}
We want to consider the contributions of the various terms in \eqref{Eq:newPXi} to the criteria \eqref{Eq:SLP} for a strong Lehmer pair.  
\subsubsection*{Central zero, imaginary part}
Consider first, in the sum over zeros $\rho^\prime$, the contribution of the unique zero $\rho_0^\prime$ which lies between $\rho_-$ and $\rho_+$.  
In the notation of \eqref{Eq:Dueneznotation}, we have
\begin{multline}\label{Eq:negligible}
\Delta^2\tim\left(\frac{1}{(\rho_--\rho_0^\prime)}+\frac{1}{(\rho_+-\rho_0^\prime)}\right)\cdot\frac{\pi}{4}=\\
\frac{-2\pi Y \left(1-4 \left(X^2+Y^2\right)/\Delta^2\right)}{1+8 \left((X^2-Y^2)/\Delta^2+16
   \left(X^2+Y^2\right)^2/\Delta^4\right)}
\end{multline}
is an odd function of $Y$.  With
\begin{equation}\label{Eq:bigY}
Y=\frac{y(\delta)}{\lambda}=\frac{\pi^2}{2\lambda^2}\left(\frac{\pi}{4}+\sum_{\rho\ne\rho_\pm}\frac{1}{t_0-\gamma}\right)\,\, \delta^2+O(\delta^4),
\end{equation}
we expect \eqref{Eq:negligible} to be negligible\footnote{and with a slight bias to be negative, since the $\pi/4$ term gives $Y$ a slight bias to be positive.}.  

\subsubsection*{Central zero, real part}
On the other hand,
with
\[
r=\frac{\beta_0^\prime-1/2}{\Delta}=\frac{x}{2\pi\delta}\sim\frac{\pi}{8}\delta+O(\delta^3)
\]
we have that
\begin{align}
-\Delta^2\tre\left(\frac{1}{(\rho_--\rho_0^\prime)^2}+\frac{1}{(\rho_+-\rho_0^\prime)^2}\right)=&\frac{8(1-4r^2)}{(1+4r^2)^2}\label{Eq:8}\\
=&8-\frac{3\pi^2}{2}\delta^2+O(\delta^4).\notag
\end{align}
Here we've made the simplifying approximation\footnote{It would be desirable to bound the error in making this approximation.} that $\gamma_0^\prime\approx t_0$, $Y\approx 0$.  This contribution is largest for $r=0$, and actually negative for $r>1/2$, see Figure \ref{F:3}.  

\subsubsection*{Stirling's formula contribution}  We must account for the contribution from both terms
$-P\Xi^\prime(\gamma_+)$ and $-P\Xi^\prime(\gamma_-)$.  With our usual notation,
\begin{equation}\label{Eq:Stirling}
\Delta^2\log(\gamma_0^\prime/2\pi)^2/2=2\pi^2\delta^2.
\end{equation}
Thus the terms \eqref{Eq:8} and \eqref{Eq:Stirling} together contribute
\begin{equation}\label{Eq:sum}
8+\frac{\pi^2}{2}\delta^2+O(\delta^4).
\end{equation}

\begin{figure}
\begin{center}
\includegraphics[scale=.75, viewport=0 0 400 220,clip]{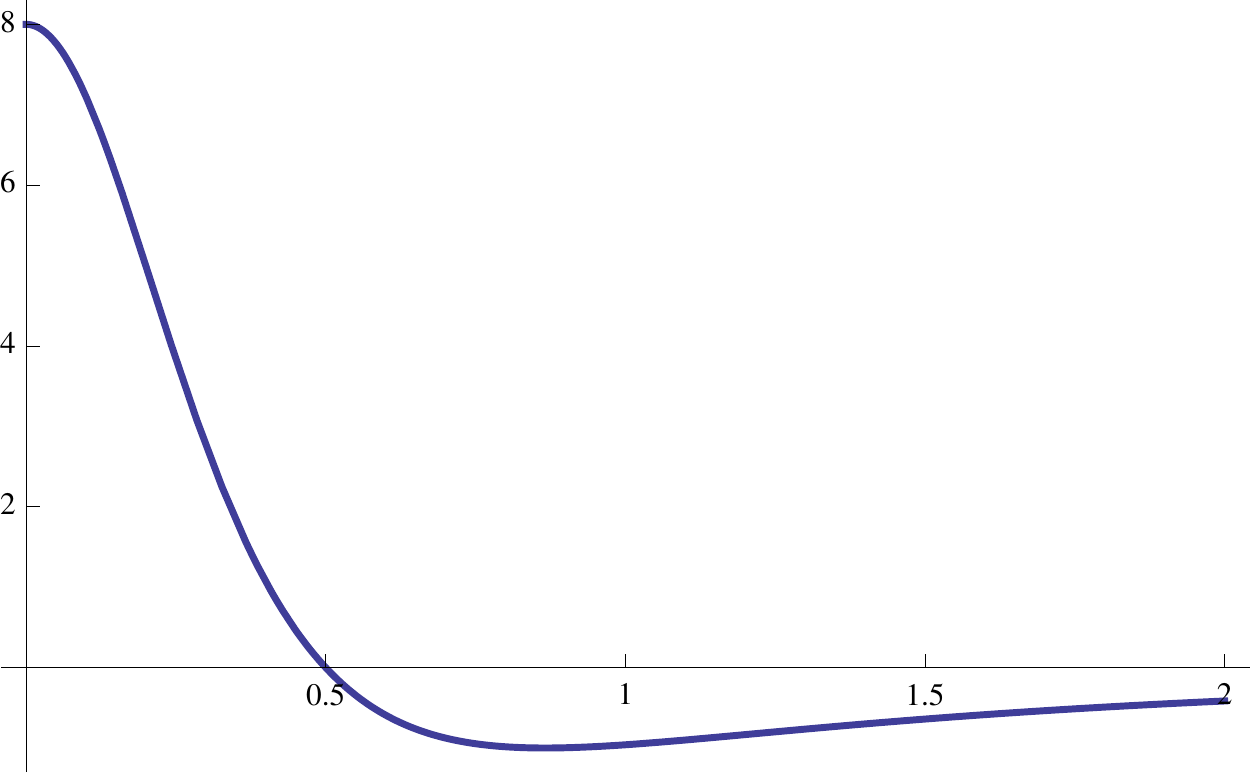}
\caption{Plot of $8(1-4r^2)/(1+4r^2)^2$.}\label{F:3}
\end{center}
\end{figure}

\subsubsection*{Remaining zeros, imaginary part}
Consider next the contribution of those zeros $\rho^\prime\ne\rho_0^\prime$.  In
\begin{equation}\label{Eq:odd}
\Delta^2 \tim\sum_{\substack{|\gamma^\prime-\gamma_0^\prime|\le 1/\log\log(\gamma_0^\prime) \\ \rho^\prime\ne\rho_0^\prime}}\left( \frac{1}{\rho_+-\rho^\prime}+\frac{1}{\rho_--\rho^\prime}\right)\cdot \frac{\pi}{4},
\end{equation}
each term is of the same form as the right side of \eqref{Eq:negligible}.  Although $Y$ is no longer small compared to $\Delta$, we expect cancellation in \eqref{Eq:odd} between those terms with $\gamma^\prime>\gamma_+$, (i.e.,\ $Y>0$) and those with $\gamma^\prime<\gamma_-$ (i.e.,\ $Y<0$).  

\begin{figure}
\begin{center}
\includegraphics[scale=1, viewport=0 0 400 270,clip]{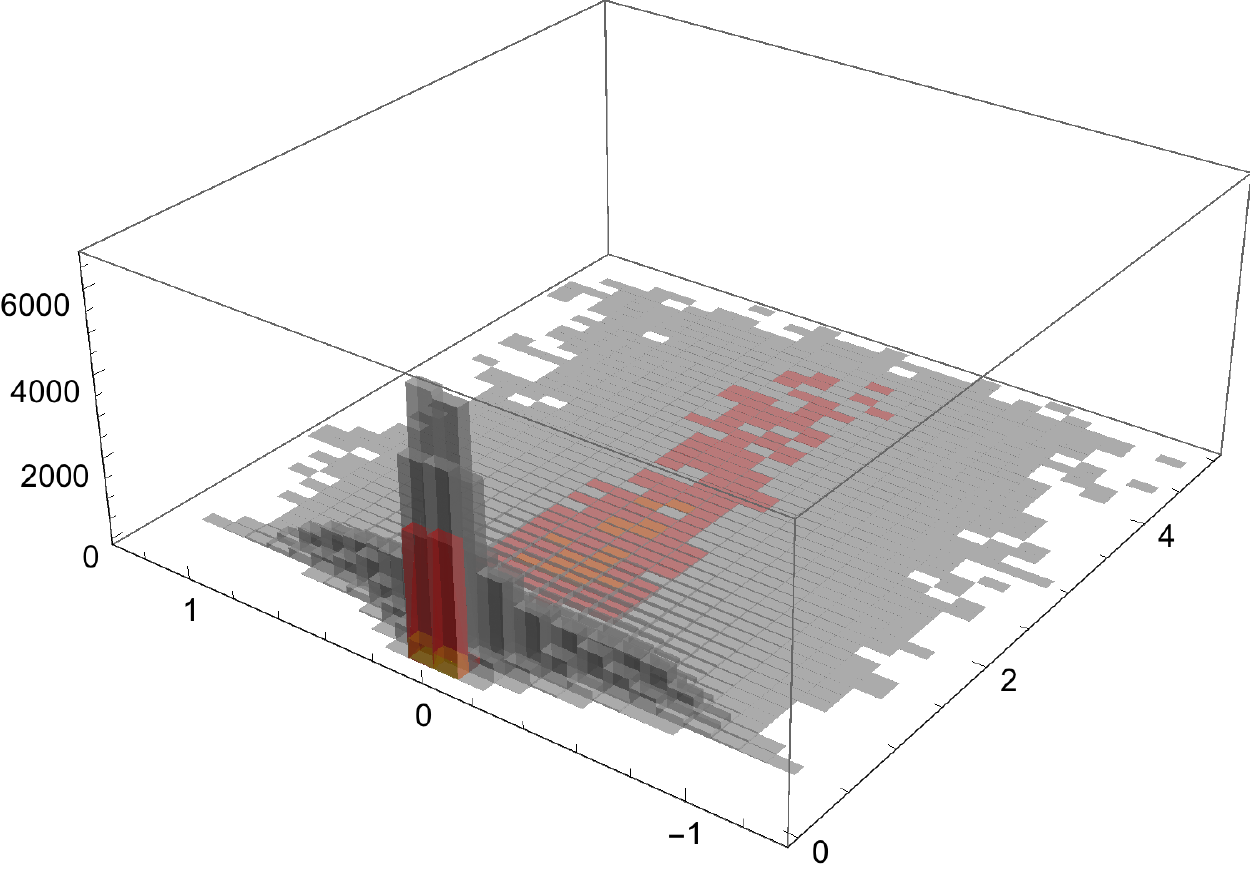}
\caption{Zeros of $\zeta^\prime(s)$, scaled by average gap between zeros of $\zeta(s)$.  Lehmer pairs in red, strong Lehmer pairs in yellow.}\label{F:4}
\end{center}
\end{figure}

We can offer a heuristic that \eqref{Eq:negligible} and \eqref{Eq:odd} are strongly correlated:  Suppose in \eqref{Eq:odd} there is an excess of zeros $\rho^\prime$ with $\gamma^\prime>\gamma^+$ compared to those with $\gamma^\prime<\gamma^-$, so that $\rho_+-\rho^\prime$ and $\rho_--\rho^\prime$ will more often have imaginary part less than zero, and their reciprocal will more often have positive imaginary part, i.e., \eqref{Eq:odd} will be positive.  In this scenario, we also expect a surplus of zeros $\rho$ of $\zeta(s)$, $\rho\ne\rho_\pm$ with $\gamma>t_0$ versus\ those with $\gamma<t_0$, so \eqref{Eq:bigY} will be negative and \eqref{Eq:negligible} will  also be positive.  (Analogously if there is an excess of zeros with $\gamma^\prime<\gamma^-$.)\ \   Figure \ref{F:12} below shows the data  comparing \eqref{Eq:negligible} with \eqref{Eq:odd}.

\begin{figure}
\begin{center}
\includegraphics[scale=1, viewport=0 0 400 270,clip]{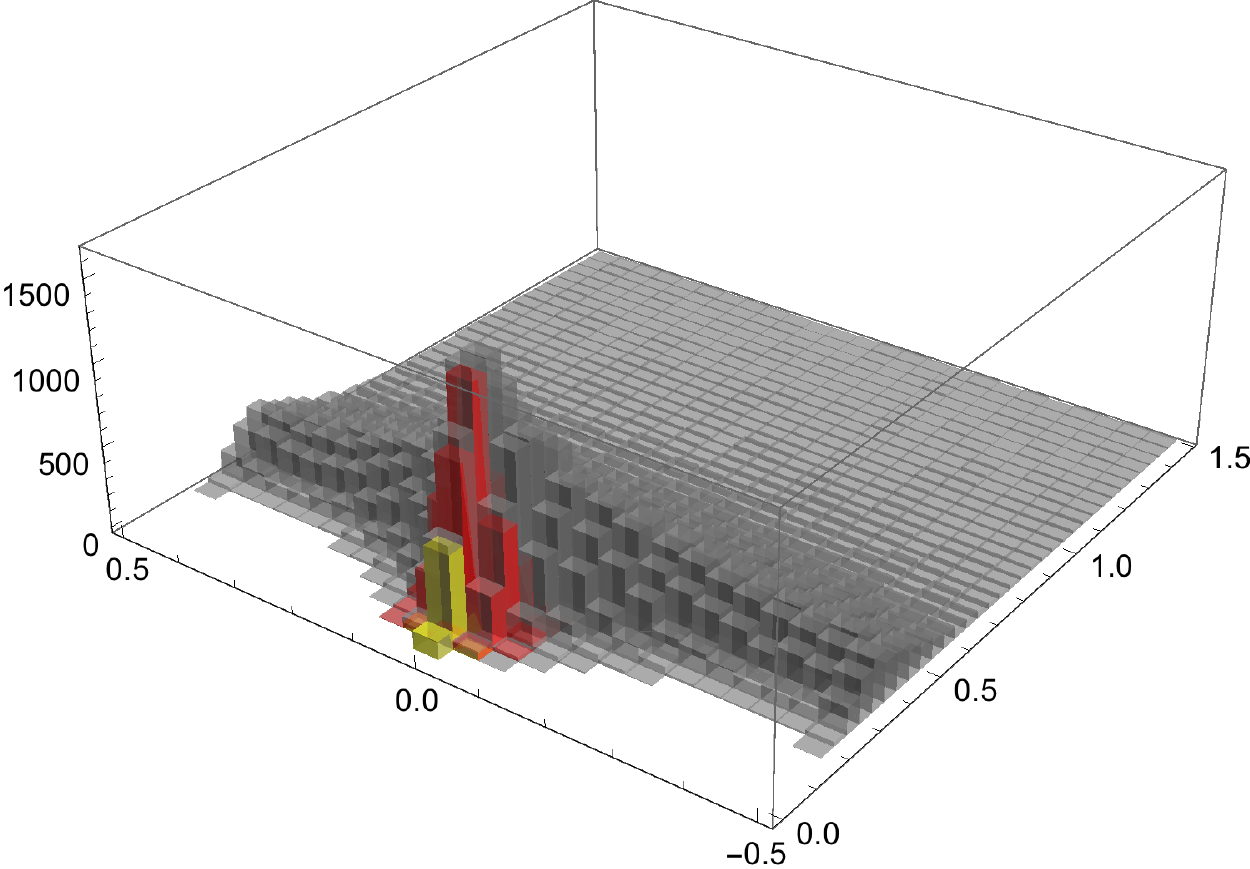}
\caption{Zeros of $\zeta^\prime(s)$, scaled by actual gap between neighboring zeros of $\zeta(s)$.  Lehmer pairs in red, strong Lehmer pairs in yellow.}\label{F:5}
\end{center}
\end{figure}

\subsubsection*{Remaining zeros, real part}  For
\begin{equation}\label{Eq:small}
-\Delta^2 \tre\sum_{\substack{|\gamma^\prime-\gamma_0^\prime|\le 1/\log\log(\gamma_0^\prime) \\ \rho^\prime\ne\rho_0^\prime}}\left( \frac{1}{(\rho_+-\rho^\prime)^2}+\frac{1}{(\rho_--\rho^\prime)^2}\right),
\end{equation}
write each term $\rho^\prime-\rho_+$ or  $\rho^\prime-\rho_-$ as $r^\prime\exp(i\theta^\prime)$, with $-\pi/2<\theta^\prime<\pi/2$.  Then
\[
-\Delta^2\tre\left(\frac{1}{(\rho^\prime-\rho_\pm)^2}\right)=-\left(\frac{\Delta}{r^\prime}\right)^2\cos(2\theta^\prime).
\]

We can model this by assuming that $\theta^\prime=\pm\pi/2$ (the worst case bound), and $r^\prime=2\pi j/\lambda$ for the $j$th term in the sum (i.e., assuming the zeros $\rho^\prime$ are perfectly regularly spaced.)\ \ In this scenario, the sum is $\delta^2\sum_j 1/j^2=O(\delta^2)$.  Although we expect this term to be small, it would be desirable to have a more rigorous analysis.  

\section{Data}\label{S:7}
\begin{figure}
\begin{center}
\includegraphics[scale=.8, viewport=0 0 400 270,clip]{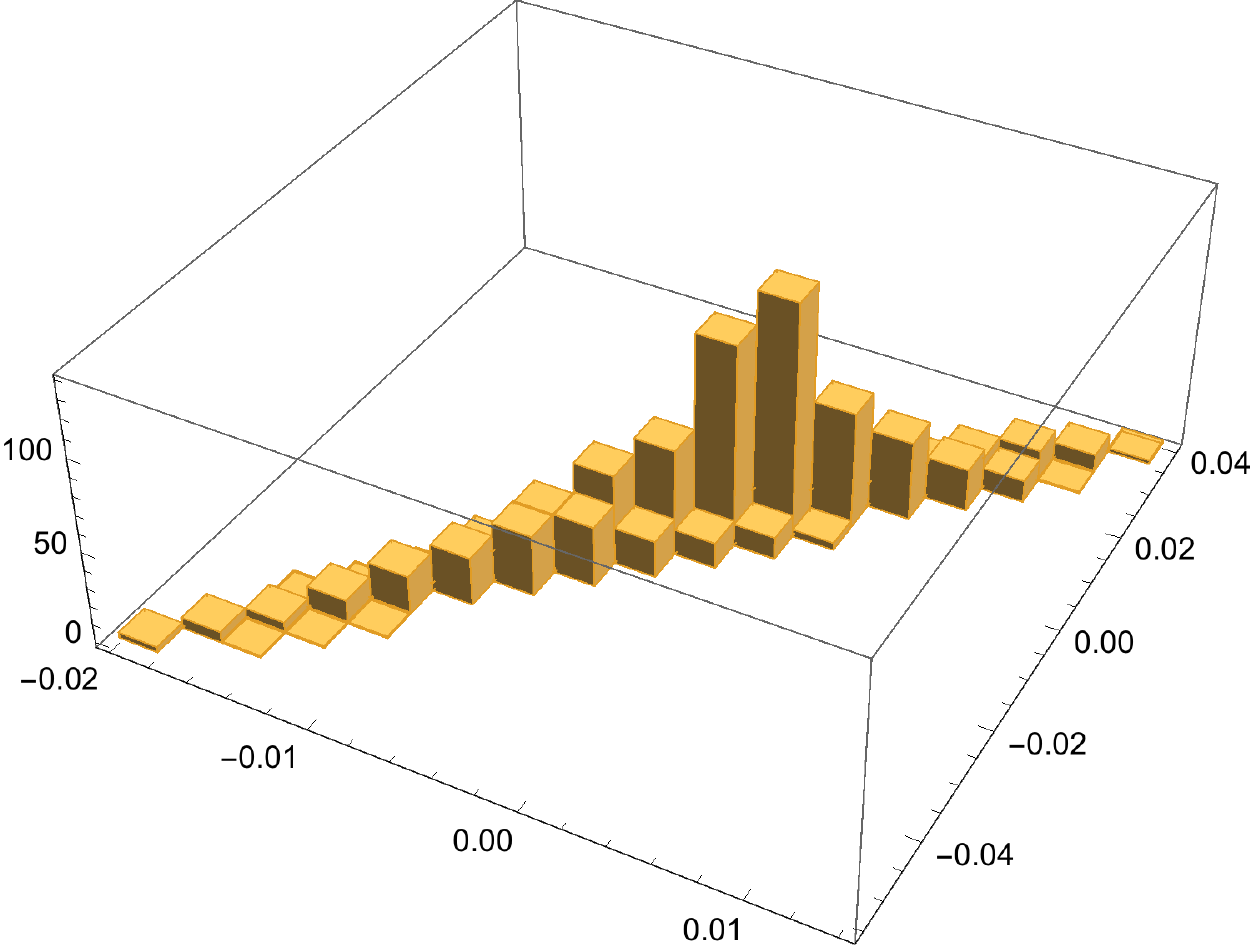}
\caption{Histogram of data for \eqref{Eq:negligible} v.\ \eqref{Eq:odd}.  \emph{Mathematica} computes the correlation as $0.99771$.}\label{F:12}
\end{center}
\end{figure}

We used \emph{Mathematica} to look for Lehmer pairs and strong Lehmer pairs in the range $10^6\le t\le 10^6+6\cdot 10^4$; i.e., $\rho_-=\rho_k$ and $\rho_+=\rho_{k+1}$ in the range $1\,747\,144\le k\le 1\,861\,805$, or $114\,661$ pairs.  Of these, 7398 were determined to be Lehmer pairs and of these, 855 were strong Lehmer pairs.  

These data were compared to the $108\,043$ zeros  of $\zeta^\prime(s)$ in the same range.  Figures \ref{F:4} and \ref{F:5} show histograms of the zeros $\rho^\prime$, shifted by $1/2+(\gamma_-+\gamma_+)/2\cdot i$.  Zeros which do not lie between a Lehmer pair are colored in gray; those which do are red or yellow according to whether the Lehmer pair is not, or is, a strong Lehmer pair.  In Figure \ref{F:4}, the zeros are scaled by $\log(\gamma^\prime)/2\pi$, the reciprocal of the \emph{average} gap between the zeros.  In Figure \ref{F:5}, the zeros are scaled by $1/(\gamma_+-\gamma_-)$, the reciprocal of the \emph{actual} gap between the zeros.  Note the different scales on the imaginary axis, which extends to the left in both figures.  

In Figure \ref{F:4} all of the data are shown.  In particular, there are infrequent zeros whose distance from the critical line (axis extends to the right) is as much as 5 times the average distance, and corresponding to both Lehmer pairs and strong Lehmer pairs.  In Figure \ref{F:5}, the data have been truncated on the real axis at those zeros which lie at $1.5$ times the actual distance, in order to highlight the behavior of those zeros closest to the critical line. (By definition the  data lie between $-0.5$ and $0.5$ on the imaginary axis in this figure.)

Figure \ref{F:12} shows, for the 855 strong Lehmer pairs,  a histogram of the contribution of \eqref{Eq:negligible} and \eqref{Eq:odd}, respectively,  to $-\Delta^2\left(P\Xi^\prime(\gamma_+)+P\Xi^\prime(\gamma_-)\right)$.   As expected these are both small and  positively correlated: \emph{Mathematica} computes the correlation as $0.99771$.

Figure \ref{F:6} shows the histogram of the terms \eqref{Eq:small} \emph{divided by} $\delta^2$.  This supports our heuristic that \eqref{Eq:small} is $O(\delta^2)$.

Figure \ref{F:8}  shows the combined contribution of the most significant terms,
\begin{gather*}
-\Delta^2\tre\left(\frac{1}{(\rho_--\rho_0^\prime)^2}+\frac{1}{(\rho_+-\rho_0^\prime)^2}\right)\sim 8-\frac{3\pi^2}{2}\delta^2,\quad\text{and}\\\
\Delta^2\log(\gamma_0^\prime/2\pi)^2/2=2\pi^2\delta^2,
\end{gather*}
respectively.  As expected these are strongly negatively correlated: \emph{Mathematica} computes the correlation as $-0.997486$.  Their combined contribution \eqref{Eq:sum} goes a long way towards explaining the location of the strong Lehmer pairs in the histogram Figure \ref{F:5}:   $8+\pi^2\delta^2/2<42/5$ when $\delta<0.285$.  In other words, strong Lehmer pairs tend to arise from a small gap between zeros of $\zeta(s)$, and from a zeros of $\zeta^\prime(s)$ very near the critical line.

We can reframe this in terms of Soundararajan's Conjecture B of \cite{Sound} which predicts:
\[
\text{(i) }\liminf \,\delta=0\quad\Leftrightarrow\quad \text{(ii) }\liminf \,x=0.
\]
\begin{conjecture}
Soundararajan's Conjecture B implies the existence of infinitely many strong Lehmer pairs, and thus, by Theorem \ref{Th:1} and the theorem of Csordas \emph{et.\ al.} in \cite{CSV}, that the de Bruijn-Newman constant $\Lambda$ is $0$.
\end{conjecture}



\begin{figure}
\begin{center}
\includegraphics[scale=.8, viewport=0 0 400 270,clip]{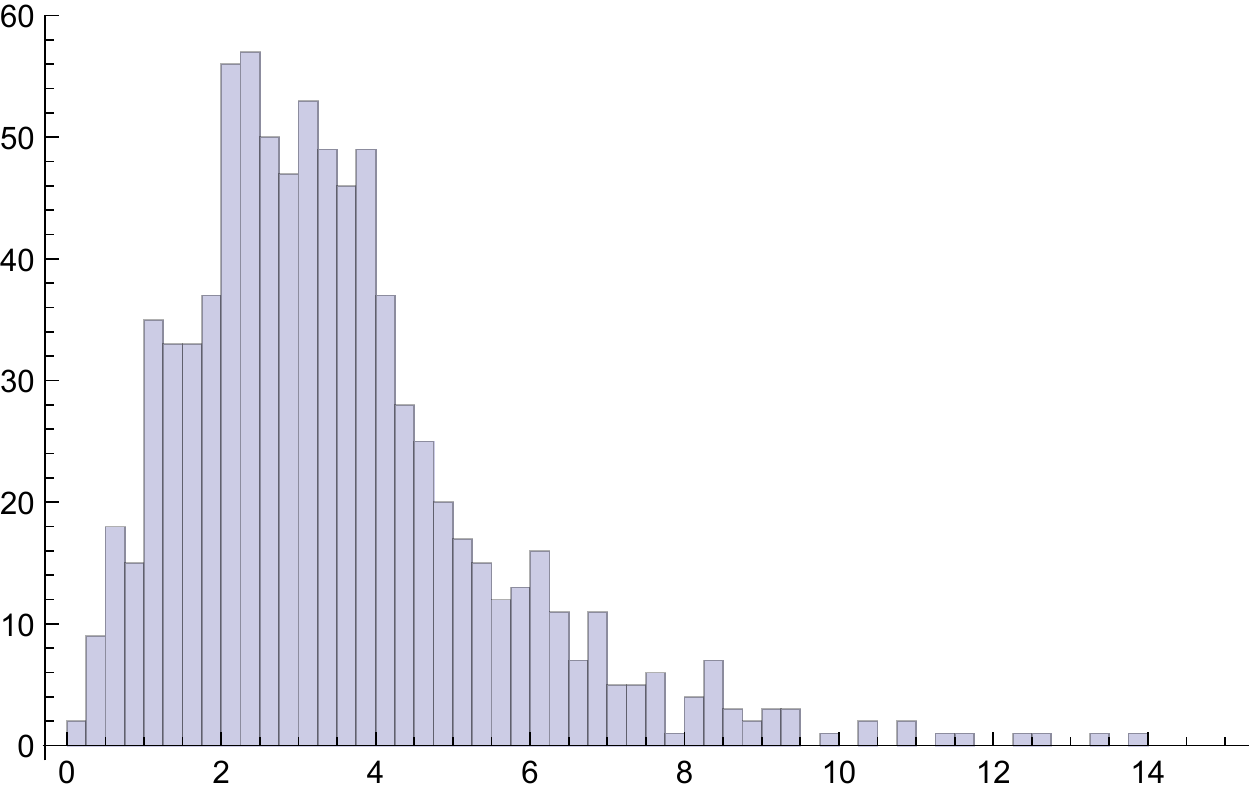}
\caption{Contribution of the terms \eqref{Eq:small}, divided by $\delta^2$}\label{F:6}
\end{center}
\end{figure}

\begin{figure}
\begin{center}
\includegraphics[scale=.8, viewport=0 0 400 270,clip]{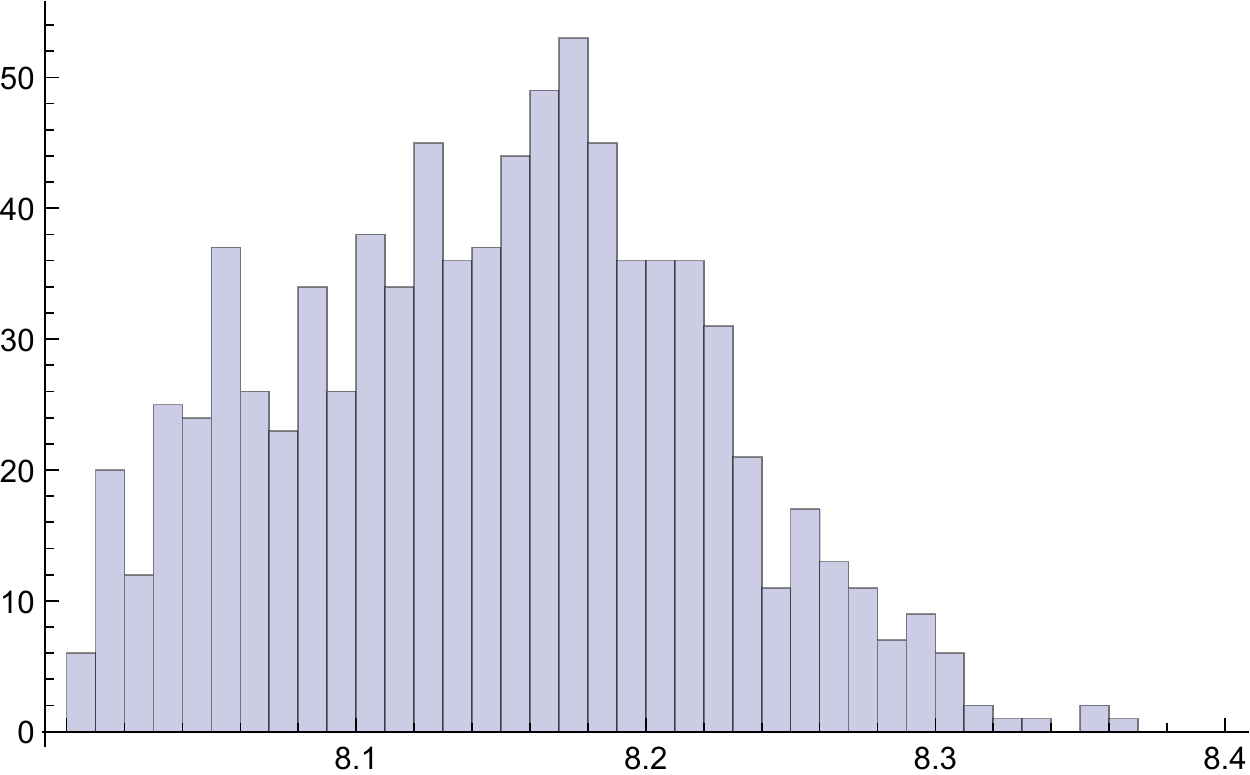}
\caption{Contribution of the combined terms \eqref{Eq:8} and \eqref{Eq:Stirling}}\label{F:8}
\end{center}
\end{figure}

\subsubsection*{Acknowledgements}  We would like to thank David Farmer for sharing his computations of zeros of $\zeta^\prime(s)$ in the range $10^6\le t\le 10^6+6\cdot 10^4$, and the anonymous referee for a careful reading of the manuscript and helpful suggestions.

\end{document}